\documentclass[square]{sigplanconf}
\pdfoutput=1
\usepackage[utf8]{inputenc}
\usepackage[T1]{fontenc}
\usepackage{microtype}
\usepackage{amsmath,amssymb,amsthm,enumerate,tikz}
\PassOptionsToPackage{hyphens}{url}\usepackage{hyperref}
\usepackage{bussproofs,xcolor,stackengine}
\hypersetup{
  colorlinks,
  linkcolor={red!70!black},
  citecolor={green!50!black},
  urlcolor={blue!50!black}
}

\newcommand{\fa}[2]{\ensuremath{\Pi(#1),\ #2}}
\newcommand{\fax}[2]{\ensuremath{\Pi#1,\ #2}}
\newcommand{\ex}[2]{\ensuremath{\Sigma(#1),\ #2}}
\newcommand{\empt}{\ensuremath{\mathbf{0}}}
\newcommand{\unit}{\ensuremath{\mathbf{1}}}
\newcommand{\bool}{\ensuremath{\mathbf{2}}}
\DeclareMathOperator{\myap}{ap}
\DeclareMathOperator{\myapd}{apd}
\DeclareMathOperator{\colim}{colim}

\DeclareMathOperator{\Id}{Id}
\DeclareMathOperator{\ind}{ind}
\DeclareMathOperator{\transport}{transport}
\DeclareMathOperator{\quotient}{quotient}
\newcommand{\id}{\textnormal{id}}
\newcommand{\istrunc}[1]{\textnormal{is-}#1\textnormal{-type}}
\newcommand{\refl}{\textnormal{refl}}
\newcommand{\base}{\textnormal{base}}
\newcommand{\lp}{\textnormal{loop}}
\newcommand{\surf}{\textnormal{surf}}
\newcommand{\isprop}{\textnormal{is-prop}}
\newcommand{\isset}{\textnormal{is-set}}
\newcommand{\prop}{\textnormal{Prop}}
\newcommand{\ap}[2]{\ensuremath{\myap_{#1}(#2)}}
\newcommand{\apd}[2]{\ensuremath{\myapd_{#1}(#2)}}
\newcommand{\eps}{\ensuremath{\varepsilon}}
\newcommand{\N}{\mathbb{N}}
\newcommand{\U}{\mathcal{U}}
\newcommand{\sy}{\ensuremath{^{-1}}}
\newcommand{\myurl}[1]{\href{#1}{\path{#1}}}

\newtheorem{theorem}{Theorem}[section]

\newtheorem{proposition}[theorem]{Proposition}
\newtheorem{lemma}[theorem]{Lemma}
\newtheorem{corollary}[theorem]{Corollary}

\theoremstyle{definition}
\newtheorem{definition}[theorem]{Definition}

\newtheorem{example}[theorem]{Example}

\theoremstyle{remark}

\usetikzlibrary{arrows}
\usetikzlibrary{calc}

\usepackage{balance}

\begin{document}

 \setlength{\pdfpageheight}{\paperheight}
\setlength{\pdfpagewidth}{\paperwidth}

\toappear{}
%% \conferenceinfo{CPP '16}{January 18--19, 2016, Saint Petersburg, Florida, USA} \copyrightyear{2016}
%% \copyrightdata{978-1-nnnn-nnnn-n/yy/mm} \doi{nnnnnnn.nnnnnnn}
%% \exclusivelicense

\title{Constructing the Propositional Truncation using Non-recursive HITs}

\authorinfo{Floris van Doorn}
           {Carnegie Mellon University}
           {fpv@andrew.cmu.edu}

\maketitle

\begin{abstract}
In homotopy type theory, we construct the propositional truncation as a colimit, using only
non-recursive higher inductive types (HITs). This is a first step towards reducing recursive HITs to
non-recursive HITs. This construction gives a characterization of functions from the propositional
truncation to an arbitrary type, extending the universal property of the propositional
truncation. We have fully formalized all the results in a new proof assistant, Lean.
\end{abstract}

 % MATHEMATICAL LOGIC AND FORMAL LANGUAGES


\section{Introduction}

Homotopy Type Theory (HoTT) is based on a connection of intensional type theory with homotopy theory
and higher category theory~\cite{Awodey2009homotopy}. In HoTT a type can be viewed as a topological
space, up to homotopy. In this setting the notion of an inductive type can be generalized to a
notion of a \emph{Higher Inductive Type} (HIT)~\cite{HoTTbook}. When defining a HIT, you can specify
not only the point constructors, but also the path constructors of a type. For example, you can
define the circle $S^1$ as a HIT with one point constructor and one path constructor, generated by:
\begin{itemize}
  \item $\base : S^1$
  \item $\lp : \base = \base$
\end{itemize}
Using the Univalence Axiom~\cite{Voevodsky2014univalence}, you can prove that $\lp \neq
\refl_\base$, which means the circle is not just the unit type.

Another HIT is the propositional truncation $\|A\|$ of a type $A$. An inhabitant of $\|A\|$
indicates that $A$ is inhabited, without specifying a particular inhabitant. The type $\|A\|$ is
always a \emph{mere proposition}, meaning that any two inhabitants are equal. The propositional
truncation is similar to the bracket type~\cite{Awodey2004Propositions} in extensional type theory
and to the squash type~\cite{Constable1986NuPRL} in NuPRL. The propositional truncation of a type
$A$ can be specified as a HIT with these constructors:
\begin{itemize}
  \item $|{-}| : A \to \|A\|$
  \item $\eps : \fa{x, y : \|A\|}{x = y}$
\end{itemize}
Note that the path constructor quantifies over all elements of the type $\|A\|$, that is, the type
we are defining. This means that it is a \emph{recursive} HIT because we can apply the path
constructor $\eps$ recursively. An example of recursive application is (loosely speaking)
\begin{equation}
  \apd{\eps(x)}{\eps(y,z)} : \eps(x,y) \cdot \eps(y,z) = \eps(x,z)\label{eq:epsilon}
\end{equation}
for $x,y,z : \|A\|$.

HITs are not very well understood. There is no general theory of HITs that specifies which HITs are
allowed, or what form the constructors can have. Giving such a theory would also require giving a
model where all such HITs exist to ensure consistency of the resulting type theory. To do it, it
might be useful to use a reductive approach. Suppose we can reduce a broad class of HITs to a few
particular HITs. In this case, we only need a model of these particular HITs to get a model of the
broader class of HITs. This is similar to the situation for inductive types in extensional type
theory. Every inductive type in extensional type theory can be reduced to $\Sigma$-types and
W-types~\cite{Dybjer1997inductivetypes}, so a model which has $\Sigma$-types and W-types has all
inductive types.

This paper is a first step in such a reductive approach. In this paper we reduce the propositional
truncation --- the prototypical recursive HIT --- to just two non-recursive HITs, the
\emph{sequential colimit} and the \emph{one-step truncation}, which we will define now.

The one-step truncation $\{A\}$ of a type $A$ is a non-recursive version of the propositional
truncation. It has the following constructors:
\begin{itemize}
  \item $f : A \to \{A\}$
  \item $e : \fa{x, y : A}{f(x) = f(y)}$
\end{itemize}
The difference between the propositional truncation and the one-step truncation is that the path
constructor of the former quantifies over all elements in the newly constructed type, while the path
constructor of the latter only quantifies over all elements of $A$. This means that $\{A\}$ only
adds a path between any two points already existing in $A$. In $\{A\}$ we do not add higher paths in
the same way as in $\|A\|$, e.g. we cannot form an equality analogous to~\eqref{eq:epsilon} in
$\{A\}$.

We can easily give the universal property of $\{A\}$. We call a function $g : A \to B$ \emph{weakly
  constant} if for all $a,a' : A$ we have $g(a) = g(a')$. The attribute \emph{weakly} comes from the
fact that we do not impose other conditions about these equality proofs (in
Section~\ref{s:consequences} we discuss some other notions of constancy). Then the maps $\{A\}\to B$
correspond exactly to the weakly constant functions from $A$ to $B$.

We can also define the (sequential) colimit as a HIT. Given a sequence of types $A : \N\to\U$ with
maps $f:\fa{n:\N}{A_n\to A_{n+1}}$ the colimit is the HIT with the following constructors:
\begin{itemize}
  \item $i : \fa{n:\N}{A_n \to \colim(A,f)}$
  \item $g : \fa{n:\N}{\fa{a:A_n}{i_{n+1}(f_n(a))=i_n(a)}}$
\end{itemize}
We will sometimes leave the arguments from $\N$ implicit for $f$, $i$ and $g$.

Our construction of the propositional truncation is as follows. Given a type $A$, we can form the
sequence
\begin{equation}
A\stackrel{f}{\to}\{A\}\stackrel{f}{\to}\{\{A\}\}\stackrel{f}{\to}\{\{\{A\}\}\}\stackrel{f}{\to}
\cdots \label{eq:sequence}
\end{equation}
Then the colimit $\{A\}_\infty$ of this sequence is the propositional truncation of $A$. What we
mean by this is that the type $\{A\}_\infty$ satisfies exactly the formation, introduction,
elimination and computation rules used to define the propositional truncation. From this we
conclude:
\begin{itemize}
\item If we work in a type theory which does not have a propositional truncation operator, then we
  can define the map the propositional truncation to be the map $A\mapsto \{A\}_\infty$.
\item If we already have a propositional truncation operator, then $\{A\}$ and $\|A\|$ are
  equivalent types.
\end{itemize}

We will give an intuition why this construction works in Section~\ref{s:intuition}. The proof that
the construction is correct (given in Section~\ref{s:proof}) uses function extensionality, which is
a consequence of the Univalence Axiom~\cite[Section 4.9]{HoTTbook}. Our construction does not
otherwise use the Univalence Axiom (UA) itself, but some related results in this paper do, and in
those case we will explicitly mention that we use UA.

This construction has multiple consequences. We already discussed this work as a starting point for
a reductive approach to HITs. Another corollary is a generalization of the universal property of the
propositional truncation. The universal property of the propositional truncation states that the
type $\|A\|\to B$ is equivalent to $A\to B$ for mere propositions $B$. However, this leaves open the
question what the type $\|A\|\to B$ is for types $B$ which are not mere propositions. The
construction in this paper answers this question: the functions in $\|A\|\to B$ are precisely the
cocones over the sequence~\eqref{eq:sequence}. A different answer to this question is given
in~\cite{Kraus2014UniversalProperty} (see Section~\ref{s:related} for a comparison). Another
corollary of the construction is the following theorem. Given a weakly constant function $A\to A$,
then $A$ has split support, meaning $\|A\|\to A$. This is a known result~\cite[Theorem
  4.5]{Kraus2014anonymousexistence}, but we give an alternative proof in
Section~\ref{s:consequences}.

We have fully formalized this construction in the Lean proof assistant. Lean~\cite{Moura2015Lean} is
a interactive theorem prover under development by Leonardo de Moura at Microsoft Research. It is
based on a version of the calculus of inductive constructions, like Coq and Agda. It has two
libraries, a standard library for constructive and classical mathematics, and a HoTT library for
homotopy type theory. In the HoTT library we are trying a new way to deal with HITs. Instead of Dan
Licata's trick~\cite{Licata2011trick} we are experimenting with the reductive approach. We have two
primitive HITs: ``quotients'' (not to be confused with set-quotients) and the $n$-truncation. From
this, we can define all other commonly used HITs. For more information about the formalization and
Lean, see Section~\ref{s:formalization}.

\section{Preliminaries}\label{s:prelim}

In this section we will present some basic definitions from~\cite{HoTTbook} used in this paper. We
follow the informal style of writing proofs from~\cite{HoTTbook}. In particular, if $f : A \to B \to
C$ is a binary function, we write $f(x,y)$ for the application $(f\ x)\ y$.

The basis of HoTT is intensional Martin L\"{o}f Type Theory. The \emph{path type} (also called the
\emph{identification type}, \emph{identity type}, \emph{equality type}) is an inductive type which
for each type $A$ and element $a : A$ gives a type family $\Id_A(a) : A \to \U$ which is generated
by $\refl_a : \Id_A(a, a)$. The type $\Id_A a b$ is also written $a =_A b$ or $a = b$. Elements $p :
a = b$ are called \emph{paths}, \emph{identifications} or \emph{equalities}. Elements of $a = a$ are
called \emph{loops}. The identity type is \emph{intensional}, which means that it can have multiple
inhabitants which are not equal. The type $a = b$ should not be confused with the judgement $a
\equiv b$ stating that $a$ and $b$ are \emph{definitionally} or \emph{judgmentally} equal, which is
a meta-theoretic concept. We can also form the path type between to paths. That is, if $p, q : a =_A
b$, then $p =_{a=_Ab} q$ is a \emph{2-dimensional path type}. In a similar manner we can define the
\emph{higher dimensional path types}. The following rule is the \emph{path induction} principle for
the identity type.

\begin{center}
\AxiomC{\stackanchor{$A : \U \qquad a : A \qquad P : \fa{b : A}{a =_A b \to \U}$} {$\rho :
    P(a,\refl_a) \qquad b : A \qquad p : a =_A b$}} \UnaryInfC{$\ind_=(P,\rho,b,p): P(b,
  p)$}\DisplayProof{}
\end{center}

Informally, path induction states that if we want to prove some property $P$ about an arbitrary
path, it is sufficient to prove it for reflexivity paths. The corresponding computation rule is
$\ind_=(P,\rho, a, \refl_{a})\equiv \rho$.

Using the path induction principle, we can define the following basic operations. Given elements $x,
y, z : A$ and paths $p : x = y$ and $q : y = z$ we define the \emph{concatenation} $p\,\cdot\,q : x
= z$ and the \emph{inverse} $p\sy : b = a$. If $f : A \to B$ is a function we can apply it to paths
to get $\ap fp : f(x) = f(y)$. If $P : A \to \U$ is given then $p$ induces a function $P(x) \to
P(y)$. In particular, if $u : P(x)$, then we have $\transport^P(p,u) : P(y)$. We also write $p_*(u)$
for $\transport^P(p,u)$. These definitions satisfy the obvious coherence laws such as
$\refl_x\,\cdot\,p = p$ and $\ap{f}{p\sy}=(\ap fp)\sy$.

For two functions $f,g : A \to B$ we write $f\sim g$ for the type of \emph{homotopies} between $f$
and $g$, which are the proofs that $f$ and $g$ are pointwise equal: $\fa{x : A}{f(x) = g(x)}$. A
function $f : A \to B$ is an \emph{equivalence} if it has a left and a right inverse, i.e. if there
are functions $g, h : B \to A$ such that $g \circ f\sim \id_A$ and $f\circ h\sim \id_B$. If $f$ is
an equivalence, then $f$ also has a two-sided inverse, written $f\sy$. The fact that $f$ is an
equivalence is written $f : A \simeq B$, and $A$ and $B$ are called \emph{equivalent} types.

If $f$ and $g$ are equal functions, then $f$ and $g$ are homotopic. This means that we have a map
$(f = g) \to (f \sim g)$. \emph{Function Extensionality} is the axiom that this map is an
equivalence. Similarly, if two types $A$ and $B$ are equal, then they are equivalent, so we have a
map $(A = B) \to (A \simeq B)$. The \emph{Univalence Axiom} states that this map is an equivalence.

The types in a universe are stratified into a hierarchy of $n$-types for $n\geq-2$. A type $A$ is a
$(-2)$-type or \emph{contractible} if it has a unique point, that is, if $\ex{x : A}{\fa{y : A}{x =
    y}}$. In this case we can prove that $A\simeq\unit$, where $\unit$ is the unit type. A type $A$
is an $(n+1)$-type if all its path types are $n$-types, i.e., if $\fa{x\ y : A}{\istrunc{n}(x =
y)}$. This hierarchy is inclusive in the sense that every $n$-type is an $(n+1)$-type. The $n$-types
are the types where the path structure becomes trivial for high dimensions, all paths above
dimension $n+2$ are trivial (inhabited by a unique element, i.e. contractible).

The $(-1)$-types are called \emph{mere propositions}, and a type is a mere proposition iff all its
inhabitants are equal, i.e. $\fa{x\ y : A}{x = y}$. The $0$-types are called \emph{sets}, and are
the types where all equality types are mere propositions, i.e. where the uniqueness of identity
proofs holds. A type $A$ is a set iff it satisfies \emph{Axiom K}: every loop is equal to
reflexivity. The $1$-types are the types where the equality types are sets, and so on.

\begin{example}\mbox{}
\begin{itemize}
\item the natural number $\N$ and the booleans $\bool$ are sets, but not mere propositions.
\item The circle $S^1$ is a 1-type, but not a set.
\item For a function $f : A \to B$, the statement ``$f$ is an equivalence'' is a mere proposition.
\item For a type $A$ and $n\geq-2$, the statement ``$A$ is an $n$-type'' is a mere proposition.
\item The 2-sphere $S^2$ is defined to be a higher inductive type with constructors $\base : S^2$
  and $\surf : \refl_\base = \refl_\base$. Note that $\surf$ is a 2-dimensional path
  constructor. $S^2$ is strongly expected not to be $n$-truncated for any $n$, but this is not yet
  proven.
\end{itemize}
\end{example}

We write $\prop:\equiv \ex{X:\U}{\isprop(X)}$ for the type of mere propositions. Given $Y:\prop$
we also write $Y$ the underlying type of $Y$.

For any type $A$ and $n\geq-1$ we can define the \emph{$n$-truncation} $\|A\|_n$ of $A$ which is
defined to have the same elements as $A$ but all $(n+1)$-dimensional paths equated. We already saw
the \emph{propositional truncation}, which is the $(-1)$-truncation of $A$, where we identify all
elements in $A$ with each other. We write $\|A\|$ for the propositional truncation. The
\emph{set-truncation} or $0$-truncation identifies all parallel paths in $A$ (two paths are said to
be parallel if they have the same type). In general, the $n$-truncation $\|A\|_n$ can be defined as
a HIT~\cite[Section 7.3]{HoTTbook}. The type $\|A\|_n$ is an $n$-type, and it is universal in the
sense that functions $A\to B$ where $B$ is an $n$-type factor through $\|A\|_n$.

A notion related to an $n$-type is an \emph{$n$-connected type}, which is a type that has trivial
path spaces below dimension $n$. Formally, $A$ is $n$-connected if $\|A\|_n$ is contractible. A type
which is $0$-connected is called \emph{connected} and a type which is $1$-connected is called
\emph{simply connected}, where the names are taken from the corresponding notions in homotopy
theory.

\section{Intuition}\label{s:intuition}
In this section we will present an intuition about the correctness of the construction. We will also
give some properties of one-step truncations.

To give an intuition why this constructions works, consider the type of booleans $\bool$ with its
two inhabitants $0,1:\bool$, and take its propositional truncation: $\|\bool\|$. Of course, this
type is equivalent to the interval (and any other contractible type), but we want to study its
structure a little more closely. The path constructor of the propositional truncation $\eps$ gives
rise to two paths between $|0|$ and $|1|$, namely $\eps(|0|,|1|)$ and $(\eps(|1|,|0|))^{-1}$. Since
the resulting type is a mere proposition, these paths must be equal. And indeed, we can explicitly
construct a path between them by using the continuity of $\eps$. To do this, we show that for every
$x : \|\bool\|$ and every $p : |0| = x$ we have $p = \eps(|0|,|0|)^{-1}\ \cdot\ \eps(|0|,x)$. We can
apply path induction on $p$, and it is trivial if $p$ is reflexivity. Since the right hand side does
not depend on $p$, this shows that any two elements in $|0| = |1|$ are equal. Note that this is just
the proof that any proposition is a set given in~\cite[Lemma 3.3.4]{HoTTbook}.

Now consider the one-step truncation of the booleans, $\{\bool\}$. This is a type with points
$a:\equiv f(0)$ and $b:\equiv f(1)$ and four basic paths, as displayed in Figure~\ref{f:tbool}. The
path constructor gives two paths from $a$ to $b$, namely $p:\equiv e(0,1)$ and
$q:\equiv(e(1,0))^{-1}$. Also, we have two loops $e(0,0) : a = a$ and $e(1,1) : b = b$. We can also
concatenate these paths to get new paths. Note that the paths $p$ and $q$ are \emph{not} equal. This
is because $p$ and $q$ are different path constructors of $\{\bool\}$, so we can use the elimination
principle of the one-step truncation to send them to any pair of parallel paths. So $p=q$ would
imply that all types are sets, contradicting the Univalence Axiom. Applying $\{{-}\}$ to $\bool$ is
the first step towards the propositional truncation: all points are equal, but these equalities are
not equal themselves.

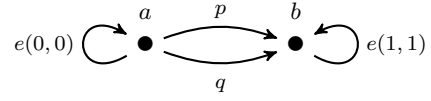
\begin{figure}
\begin{center}
\begin{tikzpicture}[>=stealth',auto,node distance=3cm,
  thick,main node/.style={font=\sffamily\Large\bfseries},text height=1.5ex] \node[main
    node,label=above:$a$] (fa) at (0,0) {$\bullet$}; \node[main node,label=above:$b$] (fb) at (2,0)
  {$\bullet$}; \path[every node/.style={font=\sffamily\small}] (fa) edge [loop left, in=150,
    out=210,looseness=8,->] node {$e(0,0)$} (fa) edge [bend left =20,->] node [above] (p){$p$} (fb)
  edge [bend right=20,->] node [below] (q){$q$} (fb) (fb) edge [loop right, in=30,
    out=-30,looseness=8,->] node {$e(1,1)$} (fb);
\end{tikzpicture}
\end{center}
\caption{The basic paths in $\{\bool\}$}
\label{f:tbool}
\end{figure}

Now we can take the one-step truncation again to get the type $\{\{\bool\}\}$. In this type we have
the points $f(a)$ and $f(b)$ and paths $\ap{f}{p}$ and $\ap{f}{q}$ between them. Now these paths
\emph{are} equal, as displayed in Figure~\ref{f:ttbool}. If we view the space $\{\{\bool\}\}$ as
taking the space $\{\bool\}$ and adding additional paths between the points, this means we added
sufficiently many paths to create a surface between the paths $p$ and $q$. The fact that $\ap{f}{p}$
and $\ap{f}{q}$ are equal is a consequence of the following lemma.

\begin{lemma}\label{l:apconstant}
  If $g:X\to Y$ is weakly constant, then for every $x, x' : X$, the function $\text{ap}_g:x=x'\to
  g(x)=g(x')$ is weakly constant. That is, $\ap gp=\ap gq$ for all $p,q:x=x'$.
\end{lemma}
\begin{proof}
  Let $q : \fa{x, y : X}{g(x)=g(y)}$ be the proof that $g$ is weakly constant, and fix $x : X$. We
  first prove that for all $y : X$ and $p : x = y$ we have
  \begin{equation}\label{e:apconstanteq}
    \ap{g}{p} = q(x,x)^{-1}\ \cdot\ q(x,y).
  \end{equation}
  This follows from path induction, because if $p$ is reflexivity, then $\ap{g}{\refl_x} \equiv
  \refl_{g(x)} = q(x,x)^{-1}\ \cdot\ q(x,x)$. The right hand side of \eqref{e:apconstanteq} does not
  depend on $p$, hence $\text{ap}_g$ is weakly constant.
\end{proof}
Since $f$ is weakly constant (as the path constructor $e$ shows), we conclude that $\ap{f}{p} =
\ap{f}{q}$. So we see that in $\{\{\bool\}\}$ the paths $\ap fp$ and $\ap fq$ are equal. More
generally, any pair of parallel paths in $A$ will be equal in $\{A\}$ after applying $f$. However,
in $\{\{\bool\}\}$ we also add \emph{new} paths between $f(a)$ and $f(b)$, for example $e(a,b)$ (see
Figure~\ref{f:ttbool}). This path is not equal to the old paths $\ap{f}{p}$ or $\ap{f}{q}$. More
generally, we have the following proposition. We will not use this proposition for the proof of
\ref{t:main}.

\begin{figure}
\begin{center}
\begin{tikzpicture}[>=stealth',auto,node distance=3cm,
  thick,main node/.style={font=\sffamily\Large\bfseries},text height=1.5ex] \node[main
    node,label=left:$f(a)$] (fa) at (0,0) {$\bullet$}; \node[main node,label=right:$f(b)$] (fb) at
  (3,0) {$\bullet$}; \path[every node/.style={font=\sffamily\small}] (fa) edge [bend left =20,->]
  node [above] (p){$\ap fp$} (fb) edge [bend right=20,->] node [below,text height=1ex] (q){$\ap fq$}
  (fb) edge [bend right=90,->] node [below] {$e(a,b)$} (fb); \draw[double,-latex,shorten
    >=1mm,shorten <=1mm] (p)--(q);
\end{tikzpicture}
\end{center}
\caption{Some paths in $\{\{\bool\}\}$.}
\label{f:ttbool}
\end{figure}

\begin{proposition}\label{p:notconnected}
For this proposition we assume the Univalence Axiom. Let $A$ be a type.
\begin{enumerate}
\item \label{e:1} For any path $p:a =_A b$ we have $\ap fp \neq e(a,b)$.
\item \label{e:2} If the type $\|\{A\}\|_1$ is a set, then it is empty. That is,
$$\isset(\|\{A\}\|_1) \to \|\{A\}\|_1 \to \empt.$$ In particular, $\{A\}$ is not simply connected.
\item $\|\{A\}\|_0\simeq \|A\|$, hence $\|\{A\}\|_0$ is a mere proposition. In particular, if $A$ is
  inhabited, then $\{A\}$ is connected.
\end{enumerate}
\end{proposition}
\begin{proof}
  \textbf{Part 1}. Assume that $\ap fp=e(a,b)$. Now define a map $h:\{A\}\to S^1$ by pattern
  matching as:
  \begin{itemize}
  \item $h(f(a)):\equiv \base$ for all $a:A$;
  \item $\ap h{e(a,b)}:=\lp$ for all $a,b:A$.
  \end{itemize}
  We compute
  \begin{align*}
    \refl_\base&=\ap{\lambda x, \base}p\\ &\equiv\ap{h\circ f}p\\ &=\ap h{\ap f
      p}\\ &=\ap{h}{e(a,b)}\\ &=\lp.
  \end{align*}
  This is a contradiction. If the universe does not contain the circle, then choose any other type
  $X$ with a point $x : X$ and a loop $p : x = x$ such that $p\neq \refl_x$. The Univalence Axiom
  ensures that such a type exists. In this case we can define $h$ similarly, mapping into $X$.

  \textbf{Part 2}. Suppose that $\|\{A\}\|_1$ is an inhabited set. We will construct an element of
  \empt. Let $z:\|\{A\}\|_1$ be an inhabitant. We can apply $\|{-}\|_1$-induction and then
  $\{{-}\}$-induction on $z$. The path constructors are satisfied automatically, since we are
  proving a mere proposition. This means that we only have to show it for point constructors, hence
  we may assume that $A$ is inhabited. So assume $x:A$.

  We have two inhabitants of $f(x)=f(x)$, namely $e(x,x)$ and $\refl_{f(x)}$. We can use the fact
  that $\|\{A\}\|_1$ is a set to show that these elements are merely equal. First we use the
  characterization of path spaces in truncated types~\cite[Theorem 7.3.12]{HoTTbook} (using UA) to
  note:
  \begin{align*}
    \|f(x)=f(x)\|_0&\simeq \left(|f(a)|_1 =_{\|\{A\}\|_1}|f(a)|_1\right).
  \end{align*}
  The right hand side is an equality type in a set, hence it is a mere proposition, and so is the
  left hand side. Now apply Theorem 7.3.12 again to get:
  \begin{align*}
    \|e(x,x)=\refl_{f(x)}\|&\simeq \left(|e(x,x)|_0=_{\|f(x)=f(x)\|_0}|\refl_{f(x)}|_0\right).
  \end{align*}
  The right hand side is an equality in a mere proposition, hence it is contractible. So the left
  hand side is also contractible, and in particular inhabited. Since we are proving a mere
  proposition, we may assume that $e(x,x)=\refl_{f(x)}$. The contradiction follows from
  part~\ref{e:1}.

  \textbf{Part 3}. To prove the equivalence, we first prove that $\|\{A\}\|_0$ is a mere
  proposition. Given $x, y : \|\{A\}\|_0$, we have to show that $x=y$. This is an equality type in a
  set, hence a mere proposition, so we may apply $\|{-}\|_0$-induction and then $\{{-}\}$-induction
  on both $x$ and $y$. The remaining goal is to show that for all $a, b : A$ we have
  $|f(a)|=|f(b)|$. This type is inhabited by $\ap{|{-}|}{e(a,b)}$.

  Now the equivalence $\|\{A\}\|_0\simeq \|A\|$ is easy, because we know that both types are mere
  propositions, hence we only need to define maps in both ways. We can define those maps easily by
  induction.
\end{proof}
We conclude that if $A$ is an inhabited type then $\{A\}$ is connected but not simply connected. So
when we apply $\{{-}\}$ to an inhabited type $A$ we add equalities in such a way that we make every
existing two points equal and any two existing parallel paths equal, but we also add new paths which
are not equal to any existing paths in $A$. Hence we have to repeat this $\omega$ many times: at
every step we kill off the existing higher equality structure, but by doing so we create new higher
equality structure. After $\omega$ many times we have killed of all the higher structure of $A$ and
are left with its propositional truncation.

\section{The Main Theorem}\label{s:proof}

Now we will prove that the construction of the propositional truncation works, in the sense that the
construction $A\mapsto \{A\}_\infty$ has the same formation, introduction, elimination and
computation rules for the propositional truncation.

We recall the definition of $\{A\}_\infty$. Given a type $A$, we define a sequence
$\{A\}_{-}:\N\to\U$ by
\begin{align}
\begin{aligned}
\{A\}_0&:\equiv A\\ \{A\}_{n+1}&:\equiv \{\{A\}_n\}
\end{aligned}
\label{e:An}
\end{align}
We have map $f_n:\equiv f : \{A\}_n\to \{A\}_{n+1}$ which is the constructor of the one-step
truncation. We define $\{A\}_\infty=\colim(\{A\}_{-},f_{-})$. This is the formation rule of the
propositional truncation (note that $\{A\}_\infty$ lives in the same universe as $A$).

We also easily get the point constructor of the propositional truncation, because that is just the
map $i_0:A\to \{A\}_\infty$. The path constructor $\fa{x, y : \{A\}_\infty}{x = y}$, i.e. the
statement that $\{A\}_\infty$ is a mere proposition, is harder to define. We will postpone this
until after we have defined the elimination and computation rules.

The elimination principle --- or induction principle --- for the propositional truncation is the
following statement. Suppose we are given a family of propositions $P : \{A\}_\infty \to \prop$ with
a section $h : \fa{a : A}{P(i_0(a))}$. We then have to construct a map $k : \fa{x :
  \{A\}_\infty}{P(x)}$. To construct $k$, take an $x : \{A\}_\infty$. Since $x$ is in a colimit, we
can apply induction on $x$. Notice that we construct an element in $P(x)$, which is a mere
proposition, so we only have to define $k$ on the point constructors. This means that we can assume
that $x\equiv i_n(a)$ for some $n : \N$ and $a:\{A\}_n$. Now we apply induction on $n$.

If $n\equiv0$, then we can choose $k(i_0(a)):\equiv h(a):P(i_0(a))$.

If $n\equiv \ell+1$ for some $\ell:\N$, we know that $a:\{\{A\}_\ell\}$, so we can induct on
$a$. The path constructor of this induction is again automatic. For the point constructor, we can
assume that $a\equiv f(b)$. In this case we need to define $k(i_{\ell+1}(f(b))) :
P(i_{\ell+1}(f(b)))$. By induction hypothesis, we have an element $y : P(i_\ell(b))$. Now we can
transport $x$ along the equality $(g_\ell(b))\sy : i_\ell(b)=i_{\ell+1}(f(b))$. This gives the
desired element in $P(i_{\ell+1}(f(b)))$.

We can write the proof in pattern matching notation:
\begin{itemize}
\item $k(i_0(a)):\equiv h(a)$
\item $k(i_{n+1}(f_n(a))):\equiv(g_n(b))_*\sy(k(i_n(b)))$
\end{itemize}
The definition $k\ (i_{0}\ a) :\equiv h\ a$ is also the judgmental computation rule for the
propositional truncation.

For the remainder of this section we will prove that $\{A\}_\infty$ is a mere proposition. We will
need the following lemma.

\begin{lemma}\label{l:pieq}
  Let $X$ be a type with $x : X$. Then the type $\fa{y:X}{x=y}$ is a mere proposition.
\end{lemma}
\begin{proof}
  To prove that $\fa{y:X}{x=y}$ is a mere proposition, we assume that it is inhabited and show that
  it is contractible. Let $f : \fa{y : X}{x = y}$. From this, we conclude that $X$ is contractible
  with center $x$. Now given any $g : \fa{y : X}{x = y}$. We know that $f$ and $g$ are pointwise
  equal, because their codomain is contractible. By function extensionality we conclude that $f=g$,
  finishing the proof.
\end{proof}

To prove that $\{A\}_\infty$ is a mere proposition, we need to show $\fa{x, y :
  \{A\}_\infty}{x=y}$. Since $\fa{y : \{A\}_\infty}{x=y}$ is a mere proposition, we can use the
induction principle for the propositional truncation on $x$, which we have just proven for
$\{A\}_\infty$. This means we only have to show that for all $a : A$ we have $\fa{y :
  \{A\}_\infty}{i_0(a)=y}$. We do not know that $i_0(a)=y$ is a mere proposition,\footnote{Of
  course, we do know that it is a mere proposition after we have finished the proof that
  $\{A\}_\infty$ is a mere proposition.} so we will just use the regular induction principle for
colimits on $y$. We then have to construct two inhabitants of the following two types:
\begin{enumerate}
\item For the point constructor we need $p(a,b) : i_0(a) = i_n(b)$ for all $a : A$ and $b :
  \{A\}_n$.
\item We have to show that $p$ respects path constructors:
  \begin{equation}
    p(a,f(b))\ \cdot\ g(b) = p(a,b).\label{e:coh}
  \end{equation}
\end{enumerate}

We have a map $f^n:A \to \{A\}_n$ defined by induction on $n$, which repeatedly applies $f$. We also
have a path $g^n(a): i_n(f^n(a)) = i_0(a)$ which is a concatenation of instances of $g$.

\begin{figure}\begin{center}
\begin{tikzpicture}[>=stealth',auto,node distance=3cm,
  thick,main node/.style={font=\sffamily\Large\bfseries},text height=2ex]

  \node[main node,label=left:$a$] (a) at (0,0) {$\bullet$}; \node[main node,label=left:$f^n(a)$]
  (fa) [above of=a] {$\bullet$}; \node[main node,label=left:$f^{n+1}(a)$] (ffa) [above of=fa]
       {$\bullet$}; \tikzset{node distance=2.5cm}; \node[main node,label=right:$b$] (b) [right
         of=fa] {$\bullet$}; \node[main node,label=right:$f(b)$] (fb) [right of=ffa] {$\bullet$};
       \tikzset{node distance=2cm}; \node (Am1) [right of=fb] {$\{A\}_{n + 1}$}; \node (Am) [right
         of=b] {$\{A\}_n$}; \node (An) [right of=a] {$A$}; \draw ($(ffa)!0.4!(fb)$) ellipse (2.7cm
       and 1cm); \draw ($(fa)!0.4!(b)$) ellipse (2.4cm and 1cm); \draw (a) circle (1cm); \path[every
         node/.style={font=\sffamily\small}] (fa) edge node [left] {$g^n$} (a) edge node [right]
       {$g$} (ffa) (a) edge [bend left=60] node [above left] {$g^{n+1}$} (ffa) (fb) edge [bend
         right=20] node [above] {$e$} (ffa) edge node [left] {$g$} (b);
\end{tikzpicture}
\caption{The definition of $p$. The applications of $i$ and the arguments of the paths are
  implicit.}
\label{f:defp}
\end{center}\end{figure}
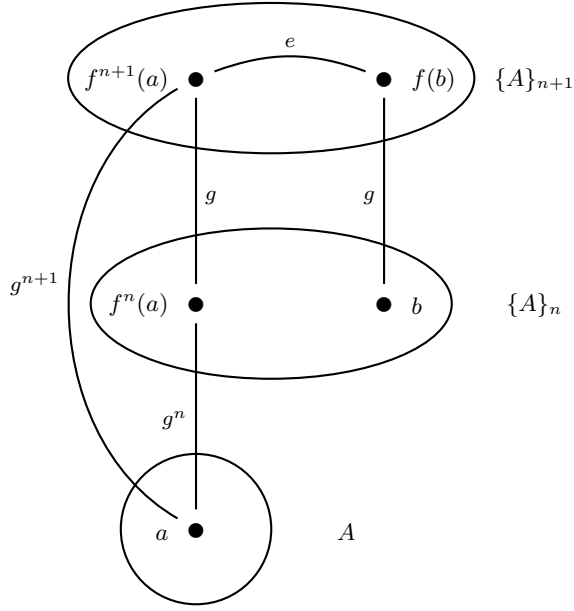

We can now define $p(a,b)$ as displayed in Figure~\ref{f:defp}, which is the concatenation
\begin{align*}
i_0(a) &= i_{n+1}(f^{n+1}(a)) &&\text{(using $g^{n+1}$)}\\ &\equiv i_{n+1}(f(f^n(a)))\\ &=
i_{n+1}(f(b)) &&\text{(using $e$)}\\ &= i_n(b) &&\text{(using $g$)}
\end{align*}

Note that by definition $g^{n+1}(a) \equiv g(f^n(a))\ \cdot\ g^n(a)$, so the triangle on the left of
Figure~\ref{f:defp} is a definitional equality.

%% $$p(a,b):\equiv (g^n(a))\sy\ \cdot\ q(f^n(a),b).$$

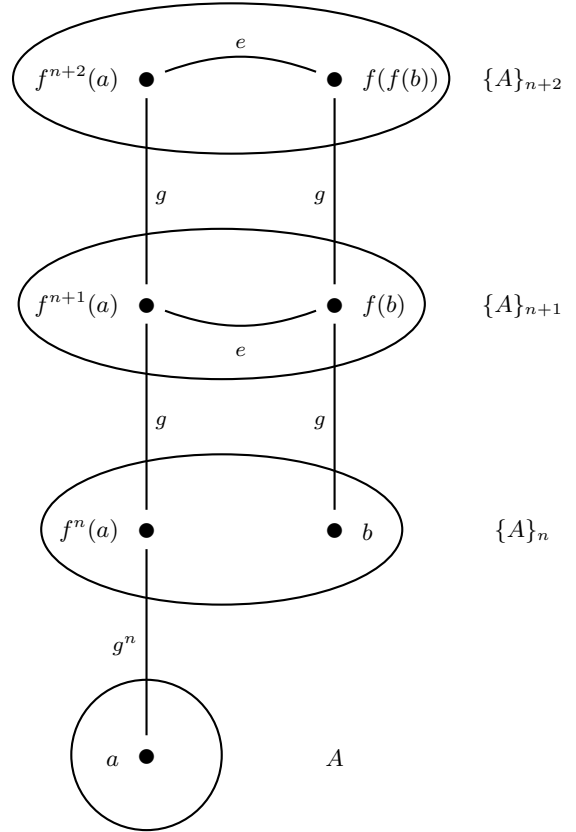
\begin{figure}\begin{center}
\begin{tikzpicture}[>=stealth',auto,node distance=3cm,
  thick,main node/.style={font=\sffamily\Large\bfseries},text height=2ex]

  \node[main node,label=left:$a$] (a) at (0,0) {$\bullet$}; \node[main node,label=left:$f^n(a)$]
  (fa) [above of=a] {$\bullet$}; \node[main node,label=left:$f^{n+1}(a)$] (ffa) [above of=fa]
       {$\bullet$}; \node[main node,label=left:$f^{n+2}(a)$] (fffa) [above of=ffa]{$\bullet$};
       \tikzset{node distance=2.5cm}; \node[main node,label=right:$b$] (b) [right of=fa]
               {$\bullet$}; \node[main node,label=right:$f(b)$] (fb) [right of=ffa] {$\bullet$};
               \node[main node,label=right:$f(f(b))$] (ffb) [right of=fffa]{$\bullet$}; \node (Am2)
                    [right of=ffb]{$\{A\}_{n + 2}$}; \node (Am1) [right of=fb] {$\{A\}_{n + 1}$};
                    \node (Am) [right of=b] {$\{A\}_n$}; \node (An) [right of=a] {$A$}; \draw
                    ($(fffa)!0.45!(ffb)$) ellipse (2.9cm and 1cm); \draw ($(ffa)!0.4!(fb)$) ellipse
                    (2.7cm and 1cm); \draw ($(fa)!0.4!(b)$) ellipse (2.4cm and 1cm); \draw (a)
                    circle (1cm); \path[every node/.style={font=\sffamily\small}] (fa) edge node
                    [left] {$g^n$} (a) edge node [right] {$g$} (ffa) (fb) edge [bend left=20] node
                    [below] {$e$} (ffa) edge node [left] {$g$} (b) edge node [left] {$g$} (ffb)
                    (fffa) edge node [right] {$g$} (ffa) edge [bend left=20] node [above] {$e$}
                    (ffb);
\end{tikzpicture}
\caption{The coherence condition for $p$. The applications of $i$ and the arguments of the paths are
  implicit.}
\label{f:cohp}
\end{center}\end{figure}

Now we have to show that this definition of $p$ respects the path constructor of the colimit, which
means that we need to show~\eqref{e:coh}. This is displayed in Figure~\ref{f:cohp}.  We only need to
fill the square in Figure~\ref{f:cohp}. To do this, we first need to generalize the statement,
because we want to apply path induction. Note that if we give the applications of $i$ explicitly,
the bottom and the top of this square are $$\ap i{e(f^{n+1}(a),f(b))}$$ and $$\ap
i{e(f^{n+2}(a),f(f(b)))},$$ respectively. This means we can apply the following lemma to prove this
equality.

\begin{figure}\begin{center}
\begin{tikzpicture}[>=stealth',auto,node distance=3cm,
  thick,main node/.style={font=\sffamily\Large\bfseries},text height=2ex]

  \node[main node,label=left:$i(x)$] (x) at (0,0) {$\bullet$}; \node[main node,label=left:$i(f(x))$]
  (fx) [above of=x] {$\bullet$}; \tikzset{node distance=2.5cm}; \node[main node,label=right:$i(y)$]
  (y) [right of=x] {$\bullet$}; \node[main node,label=right:$i(f(y))$] (fy) [right of=fx]
       {$\bullet$}; \node (Ak1) [right of=fy] {$\{A\}_{n + 1}$}; \node (Ak) [right of=y]
       {$\{A\}_n)$}; \draw ($(fx)!0.5!(fy)$) ellipse (2.8cm and 1cm); \draw ($(x)!0.5!(y)$) ellipse
       (2.6cm and 1cm); \path[every node/.style={font=\sffamily\small}] (x) edge [bend right=20]
       node [below] {$\ap ip$} (y) edge node [right] {$g$} (fx) (fy) edge [bend right=20] node
       [above] {$\ap i{p'}$} (fx) edge [bend left=20] node [below] {$\ap i{\ap fp}$} (fx) edge node
       [left] {$g$} (y);
\end{tikzpicture}
\caption{The situation in Lemma~\ref{l:cohplemma}.}
\label{f:cohplemma}
\end{center}\end{figure}
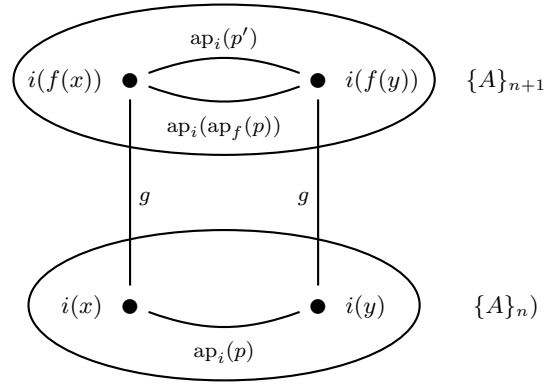

\begin{lemma}\label{l:cohplemma}
  Suppose we are given $x,y : \{A\}_n$, $p:x=y$ and $p' : f(x) = f(y)$. Then we can fill the outer
  square in Figure~\ref{f:cohplemma}, i.e.
    $$g(x)\ \cdot\ \ap ip = \ap i{p'}\ \cdot\ g(y).$$
\end{lemma}
\begin{proof}
We can fill the inner square of the diagram by induction on $p$, because if $p$ is reflexivity then
the inner square reduces to
  $$g(x)\ \cdot\ \refl_{i(x)}=\refl_{i(f(x))}\ \cdot\ g(x).$$ To show that the two paths in the top
are equal, first note that $i_{k} : \{A\}_{k} \to \{A\}_\infty$ is weakly constant. To see this,
look at Figure~\ref{f:defp}. The path from $f^n(a)$ to $b$ in that figure gives a proof of
$i_n(f^n(a))=i_n(b)$ which does not use the form of $f^n(a)$, so we also have $i_k(u)=i_k(v)$ for
$u,v :\{A\}_k$. Since $i_{n+1}$ is weakly constant, by Lemma~\ref{l:apconstant} the
function $$\text{ap}_{i_{n+1}} : f(x)=f(y)\to i_{n+1}(f(x))=i_{n+1}(f(y))$$ is also weakly
constant. This means that the two paths in the top are equal, proving the Lemma.
\end{proof}
We have now given the proof of the following theorem:

\begin{theorem}\label{t:main}
The map $A\mapsto \{A\}_\infty$ satisfies all the properties of the propositional truncation
$\|{-}\|$, including the universe level and judgmental computation rule.
\end{theorem}

\section{Consequences}\label{s:consequences}
As discussed in the introduction, we have the following immediate corollary of Theorem~\ref{t:main}.
\begin{corollary}\label{c:main}\mbox{}
\begin{itemize}
\item In a type theory with a propositional truncation operation, $\{A\}_\infty\simeq \|A\|$.
\item In a type theory without propositional truncation operation, we can define it as $\|A\|:\equiv
  \{A\}_\infty$.
\end{itemize}
\end{corollary}

In HoTT, there are multiple gradations for functions to be constant. We use the terminology
from~\cite{Shulman2015constantnessblog}.

\begin{definition}
Suppose we are given a function $f:A\to B$.
\begin{itemize}
\item $f$ is \emph{weakly constant} if $\fa{x, y : A}{f(x) = f(y)}$;
\item $f$ is \emph{conditionally constant} if it factors through $\|A\|$, i.e. if $$\ex{g:\|A\|\to
  B}{\fa{a : A}{f(a)=g(|a|)}};$$
\item $f$ is \emph{constant} if $\ex{b : B}{\fa{a : A}{f(a)=b}}$.
\end{itemize}
\end{definition}
It is not hard to see that a constant function is conditionally constant, and that a conditionally
constant function is weakly constant. Note that these definitions are not generally mere
propositions.

We can also use the universal property of the colimit to get a new universal property for the
propositional truncation, which specifies the function space $\|A\|\to B$ not only if $B$ is a mere
proposition, but for an arbitrary type $B$. Recall the definition of $\{A\}_n$ (defined in
\eqref{e:An}).

\begin{corollary}\label{c:universal} Let $A$ and $B$ be types.
\begin{enumerate}
\item We have the following universal property for the propositional truncation:
\begin{gather*}
  (\|A\|\to B)\simeq\qquad\mbox{}\\
  (\ex{h:\fax{n}{\{A\}_n\to B}}{\fax{n}{h_{n+1}\circ f \sim h_n}}).
\end{gather*}

\item For a function $k: A \to B$ we have:
\begin{gather*}
  (k\textup{ is conditionally constant})\simeq\qquad\mbox{}\\
  (\ex{h:\fax{n}{\{A\}_n\to B}}{(\fax{n}{h_{n+1}\circ f\sim h_n})\times h_0\sim k}).
\end{gather*}
\end{enumerate}
\end{corollary}
\begin{proof}
The first part is just the universal property of the colimit. The second part follows from the first
part, using some basic equivalences about $\Sigma$-types.
\end{proof}

We can use our theorem to give a proof of the following.
\begin{corollary}\label{c:hstable}
Every \emph{collapsible} type \emph{has split support}. That is, given a weakly constant function $h
: A \to A$, then there is a function $\|A\|\to A$.
\end{corollary}
\begin{proof}
The weakly constant function $h$ gives a function $\tilde h : \{A\}\to A$. The HIT $\{{-}\}$ is
functorial (just like all other HITs), so by its functorial action we get a map $\{\tilde
h\}:\{\{A\}\}\to\{A\}$, which we can compose with $\tilde h$ to get a map $\{\{A\}\}\to A$. By
induction on $n$ we get a map $k_n : \{A\}_n \to A$. Formally, we define
\begin{align*}
k_0(a)&:\equiv a\\ k_{n+1}(x)&:\equiv \tilde h(\{k_n\}(x))
\end{align*}
However, this sequence of maps does not form a cocone, because the triangles do not commute. (For
example for the first triangle we have to show $h(a)=a$ for all $a$.) But we can easily modify the
definition by postcomposing with $h$. Define $h_n:\equiv h\circ k_n : \{A\}_n\to A$. Now we get a
cocone; all triangles commute because $h$ is weakly constant. By Corollary~\ref{c:universal} we get
a map $\|A\|\to A$.
\end{proof}

\section{Formalization}\label{s:formalization}

\begin{table}
  \begin{center}
  \begin{tabular}{|c|p{0.7\columnwidth}|}
    \hline
    Lem~\ref{l:apconstant} & \texttt{weakly\_constant\_ap} \\
    Prop~\ref{p:notconnected} & \texttt{tr\_eq\_ne\_idp}\par
      \texttt{not\_inhabited\_hset\_trunc\_one\_step\_tr}\par
      \texttt{trunc\_0\_one\_step\_tr\_equiv} \\
    Lem~\ref{l:pieq} & in standard HoTT library \\
    Lem~\ref{l:cohplemma} & \texttt{ap\_f\_eq\_f} \\
    Th~\ref{t:main} & \texttt{is\_hprop\_truncX} \\
    Cor~\ref{c:main} & the definitions below \texttt{is\_hprop\_truncX}
      define propositional truncation,\par
      \texttt{trunc\_equiv} \\
    Cor~\ref{c:universal} & \texttt{elim2\_equiv} \par
      \texttt{conditionally\_constant\_equiv} \\
    Cor~\ref{c:hstable} & \texttt{has\_split\_support\_of\_is\_collapsible} \\
    \hline
  \end{tabular}
  \end{center}
\caption{Names of Theorems in the formalization}\label{t:formalization}
\end{table}
All the results in this paper have been formally verified in the proof assistant Lean. The
formalization is a single file, available at
\myurl{https://github.com/fpvandoorn/leansnippets/blob/master/cpp.hlean}. To compile the file, follow
these steps:
\begin{enumerate}
\item Download this file as \texttt{cpp.hlean}.
\item Install Lean via the instructions provided at \myurl{http://leanprover.github.io/download/}.
\item Either run \ \texttt{lean cpp.hlean} \ via your terminal, or open the file \texttt{cpp.hlean}
  in Emacs and execute it using \texttt{C-c C-x}.
\end{enumerate}

The injection from the theorems in this paper to the theorems in the formalization is given in
Table~\ref{t:formalization}. The formalization closely follows the proof presented here. Actually, a
large part the proof was first given in Lean, and found by proving successive goals given by
Lean. After the proof was given in Lean, I had to ``unformalize'' it to a paper proof. So the proof
assistant actively helped with constructing the proof. However, when unformalizing the proof, I
gained a much better insight in broader picture of the proof, the proof assistant doesn't help very
much with that goal.

The formalization heavily uses Lean's \emph{tactic proofs}. In Lean you can give a proof of a
theorem either by giving an explicit proof term (as in Agda), or by successively applying tactics to
the current goal, changing the goal (often using backwards reasoning), as in Coq. We mainly used
tactic proofs, because proofs using tactics are often shorter and quicker to write. This comes at
the cost of readability. Tactic proofs are often less readable than their declarative counterparts,
since the proof script is only one part of a ``dialogue'' between the user and the proof
assistant. The other part --- the goals given by the proof assistant --- are not given in the proof
script. However, you can request the goal state at a particular point. To do this, open the file in
Emacs, and place the point at the desired location. Now press \texttt{C-c C-g} to view the
goal. Similarly, you can see the type of a definition by moving the point on it and pressing
\texttt{C-c C-p}. For more information on how to interact with Lean, see the Lean
tutorial~\cite{Avigad2015tutorial}.

In Lean, we define the one-step truncation using the ``quotient,'' which is a primitive higher
inductive type in Lean. The quotient is the following HIT. Given a type $A$ and a type-valued
relation $R : A \to A \to \U$. Then the quotient has two constructors:
\begin{itemize}
\item $\iota : A \to \quotient_A(R)$
\item $\fa{x\ y : A}{R(x,y)\to\iota(x)=\iota(y)}$.
\end{itemize}
The quotient allows us to easily define a large class of HITs. We can define all HITs satisfying the
following conditions:
\begin{itemize}
\item It has only point and 1-path constructors.
\item All constructors are nonrecursive.
\item The path constructors don't mention the other path constructors.
\end{itemize}
In this case, we can take $A$ to be the type defined by the point constructors, and the relation $R$
generated by the path constructors. Although this list may seem restrictive, a lot of HITs still
fall into this class. The sequential colimit and one-step truncation fall in this class. Other
examples include the circle, suspensions, coequalizers and pushouts.

What's more interesting is that the quotient also defines HITs which do \emph{not} fall in this
class. In this paper we have demonstrated that the propositional truncation can be defined just
using quotients. Another class of HITs which can be defined using quotients is HITs with
2-constructors. These 2-constructors are equalities between concatenations of 1-constructors. The
exact HITs we can form is described in \cite[section HITs]{vanDoorn2015Leanblog}. This construction
allows us to construct HITs such as the torus and the groupoid quotient.

In Lean, the quotient is a primitive notion: the type former, constructors and recursor are
constants, and the computation rule for the recursor is added as a computation rule. This is also
done for one other HIT, namely $n$-truncations. Using just these two HITs, we can define all
commonly used HITs.

In Coq and Agda, HITs are usually defined using Dan Licata's trick~\cite{Licata2011trick}. To define
a HIT $X$ using this trick, you first define a normal inductive type $Y$ with the point
constructors. Then you add the path constructors of $X$ as constants/axioms to $Y$ and define the
induction principle of $X$ using the induction principle for $Y$. Now the path constructors are
inconsistent with the induction principle for $Y$, so you have to make sure to never accidentally
use the induction principle for $Y$ anymore. In Coq and Agda this can be done using private
inductive types, which means that the induction principle $Y$ is hidden outside the module where $Y$
is defined. Outside the module it's not visible that the HIT $X$ was defined in an inconsistent
way. The disadvantage of this way is that the implementation is inconsistent. This is different in
Lean, where we only add two HITs to the type theory, which can be consistently added.

\section{Conclusion, Related and Future Work}\label{s:related}

This construction of the propositional truncation as a colimit of types is a promising method to
reduce recursive HITs to non-recursive HITs. It might be possible to use a similar method to
construct more general HITs, such as the $n$-truncation, or more generally,
localizations~\cite{Shulman2011Localizationblog}. Another HIT which may be reducable to quotients is
the W-suspensions, as defined in~\cite{Sojakova2015HITinitial}. For the $n$-truncation the
construction given here can easily be generalized, but the proof that the resulting colimit is
$n$-truncated does not seem to generalize, and seems to require new ideas.

Corollary~\ref{c:hstable} was already known, and appears in~\cite[Theorem 3]{kraus2013hedberg} and
\cite[Theorem 4.5]{Kraus2014anonymousexistence}. These known proofs show that the type of fixed
points of a weakly constant endofunction $A \to A$ is a mere proposition, but there are multiple
proofs of this fact. Two of them are given in \cite[Lemma 4.1]{Kraus2014anonymousexistence}. Also, Lemma~\ref{l:apconstant} already occurs in~\cite[Proposition 3]{kraus2013hedberg}.

The generalized universal property as written in Corollary~\ref{c:universal} is new. It is a
promising theorem, which simplifies the construction of functions from the propositional truncation
to a type which is not truncated. One application of this universal property is already given by
Corollary~\ref{c:hstable}. This proof doesn't require coming up with a suitable proposition which is
an intermediate step between $\|A\|$ and $A$. The other proofs used as intermediate step the type of
fixed points of the endofunction.

Corollary~\ref{c:universal} is closely related to the main result in
\cite{Kraus2014UniversalProperty}. Their main result also gives condition which is equivalent to
finding a map $\|A\|\to B$ for an arbitrary type. The advantage of our universal property is that it
can be formulated internal to a type theory and that it has been formalized in the proof assistant
Lean. Kraus' universal property can only be formulated in a type theory which has certain Reedy
limits.

On the other hand, the advantage of Kraus' universal property over the one presented here is that it
reduces to a simpler condition when defining a map $\|A\|\to B$ if it's known that $B$ is an
$n$-type. In that case, their condition becomes giving only finitely many higher paths, which can be
formulated inside type theory without Reedy limits. Our universal property doesn't simplify given
that $B$ is an $n$-type. It may be possible to modify the colimit, so that the $n$-th term in the
sequence is $n$-connected. In that case, elimination to an $n$-type requires only finitely much
information, because at some point the cocone becomes trivial.

Another question which involves Corollary~\ref{c:universal} is whether it is possible to formulate a
similar universal property without defining the sequence $\{A\}_n$. This might simplify the
universal property.

In Section~\ref{s:intuition} we have given some properties of the one-step truncation. However, we
haven't given an exact characterization. It is not hard to see that the one-step truncation of a set
with $n$ elements is a wedge of $n^2-n+1$ circles.\footnote{The one-step truncation adds $n^2$ many
  paths, and we need to collapse $n-1$ of those paths to reduce the $n$ points to a single
  point. The result is a wedge of $n^2-(n-1)$ circles.} However, it is not clear whether the
structure of the one-step truncation of other types (which are not sets) can be described in more
simple terms, for example if $\{\{\bool\}\}$ can be described in simpler terms.

\acks

I would like to thank Egbert Rijke for the helpful discussions when searching for the proof in
Section~\ref{s:proof}. I have written a post with this result on the HoTT blog\footnote{available at
  \myurl{http://homotopytypetheory.org/2015/07/28/constructing-the-propositional-truncation-using-nonrecursive-hits/}}
and I would like to all those who wrote comments to the blog for their helpful ideas and input. I
would especially like to thank Nicolai Kraus for a big simplification of the proof presented in the
blog post. I would also like to thank Mike Shulman for ideas which led to
Proposition~\ref{p:notconnected} and Martin Escardo for discovering some of the
Corollaries. Finally, I would like to thank Leonardo de Moura and Jeremy Avigad for the countless
discussions and support for Lean-related issues.

\bibliographystyle{amsalphaurl}

\balance
\bibliography{proptrunc}

% The bibliography should be embedded for final submission.

%% \begin{thebibliography}{}
%% \softraggedright

%% \bibitem[Smith et~al.(2009)Smith, Jones]{smith02}
%% P. Q. Smith, and X. Y. Jones. ...reference text...

%% \end{thebibliography}

\end{document}